\documentclass[12pt]{amsart}
\usepackage{latexsym,amssymb}

\newcounter{prop}
\newcounter{def}

\newtheorem{thm}{Theorem}
\newtheorem{cor}{Corollary}
\newtheorem{lem}{Lemma}

\newtheorem{prop}[prop]{Proposition}
\newtheorem{defi}[def]{Definition}

\newcommand{\norm}[2][]{\| \, {#2} \,\|_{#1}}
\newcommand{\abs}[1]{|{#1}|}

\newcommand{\NN}{\ensuremath{\mathbb{N}}}

\newcommand{\ZZ}{\ensuremath{\mathbb{Z}}}

\newcommand{\field}[1]{\mathbb{#1}}
\newcommand{\bZ}{\field{Z}}        
\def\zd{\bZ^d}

\newcommand{\cA}{\mathcal{A}}
\newcommand{\cB}{\mathcal{B}}
\newcommand{\cL}{\mathcal{L}}
\newcommand{\fif}{if and only if}

\newcommand{\spr}[2][]{sp_{#1}(#2)}

\newcommand{\projtensor}{\hat{\otimes}}
\newcommand{\conv}{\star}
\newcommand{\card}[1]{\mbox{card}({#1})}
\newcommand{\sprad}[2][]{r_{#1}(#2)}
\newcommand{\w}{\ensuremath{\omega}}
\newcommand{\grs}{\ensuremath{\mbox{GRS}}}
\newcommand{\ugrs}{\ensuremath{\mbox{UGRS}}}

\newcommand{\lw}[2]{\ensuremath{\ell ^1({#1}{#2})}}

\begin{document}

\begin{abstract}
We study infinite matrices $A$ indexed by a discrete group $G$ that are
dominated by a convolution operator in the sense that $|(Ac)(x)| \leq
(a \ast |c|)(x)$ for $x\in G$ and some $a\in \ell ^1(G)$. This class
of ``convolution-dominated'' matrices forms a Banach-$*$-algebra
contained in the algebra of bounded  operators on $\ell ^2(G)$. Our
main result shows that the inverse of a convolution-dominated matrix
is again convolution-dominated, provided that $G$ is amenable and
rigidly symmetric. For abelian groups this  result goes
back to Gohberg, Baskakov, and others, for non-abelian groups
completely different techniques are required, such as generalized 
 $L^1$-algebras and the symmetry of group algebras.
\end{abstract}

\title[Convolution-Dominated Operators]{ Convolution-Dominated
  Operators on Discrete Groups}

\author{Gero Fendler}
\address{Finstertal 16, D-69514 Laudenbach, Germany}
\email{gero.fendler@t-online.de}
\author{Karlheinz Gr\"ochenig}
\address{Fakult\"at f\"ur Mathematik, Universit\"at Wien, Nordbergstrasse 15,
  A-1090 wien, Austria}
\email{karlheinz.groechenig@univie.ac.at}
\author{Michael Leinert}
\address{Institut f\"ur Angewandte Mathematik, Universit\"at Heidelberg,
Im Neuenheimer Feld 294, D-69120 Heidelberg, Germany}
\email{leinert@math.uni-heidelberg.de}
\date{\today}
\thanks{K.~G.~was supported by the Marie-Curie Excellence Grant MEXT-CT 2004-517154}
\maketitle

\section{Introduction}

Is the  off-diagonal decay of an infinite  matrix inherited by its inverse matrix?
This question arises in many problems in  numerical
analysis and approximation theory and its solution has many
applications in  frame theory and
pseudodifferential operators and wireless communications.
See~\cite{DMS84,gro04,gro06,jaffard90,Sjo95,str06} for a sample of papers.

The study of the  off-diagonal decay has two distinct facets, namely the
rate  of the off-diagonal decay and the nature of the underlying index
set.  Usually the 
index set is (a subset of)  $\zd $ and  the  focus is on obtaining
various forms of off-diagonal decay  
conditions. For instance, it is 
known that polynomial decay  and subexponential decay are
preserved under matrix inversion~\cite{jaffard90,GL04a}.  

In general,  the preservation of  off-diagonal decay
under inversion depends also  on the index set.  For instance, in the theory of Calder\`on-Zygmund
operators, the index set consists of all dyadic cubes. On this index
set the quality of   the off-diagonal decay is not necessarily
preserved, and as a consequence the inverse of a Calder\`on-Zygmund
operator need not be a Calder\`on-Zygmund
operator~\cite{meyer-2,Tch86}.  

Thus the interaction between the precise form of off-diagonal decay
and the index set plays a decisive role.  This observation is  
implicit in ~\cite{jaffard90,GL04a,Sun07a}.  In~\cite{jaffard90,GL04a} it
was mentioned (without explicit proof)  that polynomial or
subexponential decay are preserved under inversion whenever  the index set of 
the matrix class  possesses a metric with a polynomial growth
condition.

We study the interaction between the decay conditions and the index
set in the context of non-commutative harmonic analysis. Precisely,
the index set will be a discrete (non-Abelian) group, e.g.,~a finitely
generated discrete group of polynomial growth.   We then 
investigate the class of convolution-dominated matrices, which are
described by a specific  type of off-diagonal decay. Convolution-dominated
matrices over the index set $\zd $  were introduced by Gohberg,
Kashoeck, and Woerdeman~\cite{GKW89} 
as a generalization of Toeplitz matrices, and they showed that this
class of matrices was closed under inversion. Similar results and
generalizations were obtained independently by Kurbatov~\cite{kur90},
Baskakov~\cite{Bas90}. Sometime later  Sj\"ostrand~\cite{Sjo95} rediscovered
their results, gave a completely different proof, and used it  in the
context of a deep theorem about 
pseudodifferential operators.  

 We consider matrices  indexed
by a discrete group $G$:  every  operator on $\ell ^2(G) $ is
described by a matrix $A$ with entries $A(x,y), x,y\in G$ by the usual
action $(Ac)(x) = \sum _{y\in G} A(x,y) c(y)$ on a sequence $c\in \ell
^2(G)$. We will consider mostly groups of polynomial growth. A
finitely generated group is of polynomial growth, if there exists a
finite set $U\subseteq G$, such that $\bigcup _{n=1}^\infty U^n =G $
and $\mathrm{card} \, U^n \leq C n^D$ for some constants $C,D>0$. 
Our main theorem reads as follows. 

\begin{thm}
  \label{main}
Let $G$ be a discrete finitely generated group of polynomial
growth. If a matrix $A$ indexed by $G$ satisfies the off-diagonal
decay condition 
$|A(x,y)| \leq a(xy^{-1}) , x,y\in G$ for some $a\in \ell ^1(G)$ and $A$ is
invertible on $\ell ^2(G)$, then there exists  $b\in \ell ^1(G)$ such
that $|A^{-1}(x,y)| \leq b(xy^{-1}), x,y\in G$. 
\end{thm}

We will extend this result and also  consider the situation  where $\ell
^1(G)$ is replaced by the weighted algebra $\ell ^1(G,\omega )$ for
certain weight functions on $G$. This weighted case is easier and 
follows from Theorem~\ref{main} by standard methods.

To put Theorem~\ref{main} into a bigger context, let us  consider the 
case $A(x,y) = a(xy^{-1}) $ for a sequence $a\in \ell ^1(G)$. This 
matrix $A$ corresponds to the convolution operator $Ac = a \ast c$ on
$\ell ^2(G)$.  Even this case is highly
non-trivial. Theorem~\ref{main} implies  the symmetry of the group
algebra $\ell ^1(G)$, i.e., the spectrum of positive elements $a^*
\ast a $ is contained in $[0,\infty )$ for all $ a\in \ell
^1(G)$. This fact is of course well known, but its proof requires the
combination of two landmark results of harmonic analysis, namely  
Gromov's characterization of finitely 
generated groups of polynomial growth as finite extensions of
nilpotent groups and Hulanicki's result that discrete  nilpotent
groups are  symmetric~\cite{grom81,hul72}.

Convolution-dominated matrices on groups of polynomial growth occur
implicitly in Sun's remarkable work~\cite{Sun07a}. His conditions on the
off-diagonal decay are somewhat complicated and exclude the basic case
of $\ell ^1$-decay. In view of the relation with the symmetry of
groups, this omission is not surprising.  

If $G$ is a discrete Abelian group $G$, the proof of the main theorem is
based on an idea of de Leeuw~\cite{Lew75}: to every operator on $G$ one can
assign an operator-valued Fourier series and then classical Fourier
series arguments, such as Wiener's Lemma, can be applied. 
   This approach is championed in~\cite{Bas90,GKW89}. 

For a non-Abelian group as the  index set, these ideas break down
completely, and a new approach  is required. Our key idea is to
replace  the Fourier series  arguments by methods taken from 
Leptin's investigation of generalized
$L^1$-algebras~\cite{lep67,lep67a,lep68}. The main insight 
is that  the algebra of convolution-dominated
matrices can be identified with a generalized $L^1$-algebra in the
sense of Leptin.   
This observation  allows us to translate the original problem
about matrix inversion  into a problem of abstract harmonic
analysis. The analysis of generalized $L^1$-algebras was  advanced by
Leptin and Poguntke~\cite{lep67,lep67a,lep68,LP79}  
and has produced deep results. In fact, we will 
resort to their  representation theoretic results and to   the
concept of the ``rigid symmetry'' of Banach algebras and 
apply these  at a crucial point.

 The relation between a ``simple'' matrix problem and the theory
of generalized $L^1$-algebras may seem surprising at first glance, but
it is  exactly this connection that allows us to use the power of
non-commutative harmonic analysis  to solve the problem. 

Let us mention that a similar theory can be established for
convolution-domi\-nated integral operators.  This generalization  
is more technical  and will be dealt with in a subsequent paper.

The paper is organized as follows: in Section~2 we give a formal
definition of the algebra of convolution-dominated operators and
identify it as a generalized $L^1$-algebra. In Section~3 we prove the
symmetry of this algebra, and in Section~4 we treat the related
concept of inverse-closedness. In particular, we prove
Theorem~\ref{main}. In Section~5 we treat the weighted case and
characterize all weights for which the generalized weighted
$L^1$-algebra is symmetric.

\textbf{Acknowledgement:} We  would like to thank Marc Rieffel for
his useful comments and questions on an early draft of the paper. 

\section{The Algebra of Convolution-Dominated Operators as a Twisted
  $L^1$-Algebra}\label{sec:2}
Let $G$ be a discrete group. 
For $x\in G$ we  denote the operator of left translation
 on $\ell ^1(G)$ and on $\ell ^2(G)$ by $\lambda(x)$, i.e. if $f\in
 \ell ^1(G)$ or  $f\in \ell ^2(G)$, then   
$\lambda(x)f\,(y) = f(x^{-1}y),\; x,y\in G$. By $\mathcal{B}(\ell ^2(G))$ we denote the algebra of bounded
 operators on $\ell ^2(G)$.
\par
For an operator $A: \ell ^2(G) \to \ell ^2(G)$ 
let
$A(x,y) = <A\delta_y,\delta_x>,\;x,y\in G$,  
be its matrix,
where $<~,~>$ is the inner  product of the Hilbert space $\ell ^2(G)$
and $\delta _x(x) = 1$ and $\delta _x(z) = 0$ for $z\neq x$. 

\begin{defi} The operator $A$ is called convolution-dominated, in
  short notation $A\in
  CD(G)$,  if there exists a sequence 
  $a\in \ell ^1(G)$ such that 
\[
\abs{A(x,y)} \leq a(xy^{-1}), \quad \forall x,y\in G.
\]
We define the  norm of $A$  as an element in $CD(G)$ by 
\[
\norm[1]{A} := \inf \{ \norm[\ell ^1]{a} : a\in \ell ^1(G),\; \abs{A(x,y)} \leq
a(xy^{-1})\;\forall x,y\in G\}.
\]
\end{defi}
 By choosing $a(z)$ to be  the supremum of the entries
of $A$ on the $z$-th diagonal, namely  $a(z) =
\sup_{\{x,y:\,xy^{-1}=z\}} \abs{A(x,y)} $,  we see that 
\[
\|A\|_1 = \sum_{z\in G} \quad \sup_{\{x,y:\,xy^{-1}=z\}} \abs{A(x,y)} \; < \;
\infty.
\]

To shed light on this definition, consider the action of the operator
or matrix $A$ on a finitely supported vector $c$ and take absolute
values: 
\begin{equation}
  \label{eq:cc28}
|(Ac)(x)| = | \sum _{y\in G} A(x,y) c(y)| \leq \sum _{y\in G}
a(xy^{-1}) \, |c(y)| = (a \ast |c|)(x) \, .
\end{equation}
Thus $A$ is dominated pointwise by a convolution operator,
whence our terminology. Clearly, if $a \in \ell ^1(G)$, i.e., if $A\in
CD(G)$, then $A$ is bounded on $\ell ^2(G)$, and the operator norm on
$\ell ^2(G)$, in fact on all $\ell ^p(G), 1\leq p \leq \infty$,   is
majorized by the $\|A\|_1$-norm. If we consider the composition of two
convolution-dominated operators $A$ and $B$, then we obtain similarly
$$
|(ABc)(x)| \leq (a \ast b \ast |c|)(x) \, ,
$$
and therefore the operator $AB $ is again convolution-dominated and we
obtain that $\|AB\|_1 \leq \|A \|_1 \, \|B\|_1$, because  $\ell ^1(G)$
is a convolution algebra. 
We may  summarize these observations  as  follows.

\par
\begin{lem}
The space $  CD(G)$ is a Banach $*$-algebra with respect to composition of
operators and taking the adjoint operator as involution. Moreover,
$CD(G)$ is continuously embedded into $\mathcal{B}(\ell ^2(G))$.  
\end{lem}

\par

Our first goal is to represent $CD(G)$ as a generalized $L^1$-algebra
in the sense of Leptin~\cite{lep67}. Consider the  $C^{\ast}$-algebra  $\ell ^{\infty}(G)$ with pointwise
multiplication and complex conjugation as involution. This algebra is  
isometrically represented as an algebra of  multiplication operators on $\ell ^2(G)$ 
by 
\[D^mf(x)=m(x)f(x),\; \mbox{ where }x\in G, f\in \ell ^2(G) ,m\in \ell
^{\infty}(G).\] 
 Analogously, we define an operator $D^m_z$ by 
\[D^m_z = \lambda(z)\circ D^m.\]
As is easily seen, the matrix of $D^m_z$ has the entries 
\begin{equation}
  \label{eq:c1}
  D^m_z(x,y)=m(y)\delta_z(xy^{-1})\, .
\end{equation}
Whereas the matrix of the multiplication operator $D^m$ is a diagonal
matrix, the matrix of $D^m_z$ is non-zero only on the $z$-th
side-diagonal. Since   every matrix can be written  as the  sum of its 
side-diagonals, every operator is  a sum of the
elementary operators $D^m_z$. This simple observation is crucial for
the analysis of convolution-dominated operators. 

Next we study  how the operators $D^m_z$  behave under composition: if
$v,w \in G$ and  $m,n\in \ell ^\infty (G)$, then 
\begin{eqnarray}
(D^n_v\circ D^m_w)(x,y) &=& \sum _{z\in G} D^n_v(x,z)D^m_w(z,y) \notag\\
&=&\sum _{z\in G} n(z) \delta_v (xz^{-1}) m(y) \delta_w (zy^{-1}) \notag\\
&=&\sum _{z\in G} n(zy) \delta_v (xy^{-1}z^{-1}) m(y) \delta_w (z) \label{eq:cc5}\\
&=&n(wy)  m(y) \delta_v (xy^{-1}w^{-1}) \notag \\
&=&n(wy)  m(y) \delta_{vw} (xy^{-1})\notag \\
&=& D^{(T_{w^{-1}}n)\, m}_{vw}(x,y) \, . \notag
\end{eqnarray}
In the last equality  we have set  $T_y n (z) = n(y^{-1}z)$ for $ n\in
\ell ^{\infty}(G)$.  We use a  notation different from $\lambda $,
because $T_y :\ell ^{\infty}(G) \to \ell ^{\infty}(G)$ is a  $C^\ast$
automorphism of the algebra $\ell ^{\infty}(G)$ and  the mapping
$y\mapsto T_y$ defines  a homomorphism of the group $G$ into the group
of 
$C^\ast$-automorphisms of $\ell ^{\infty}(G)$.
Using  this homomorphism,  we may now  form the twisted $L^1$-algebra
$\mathcal{L}=\ell ^1(G,\ell ^{\infty}(G),T)$ in
the sense of Leptin~\cite{lep67,lep67a,lep68}. The  underlying Banach
space of $\mathcal{L}$ is  the space of $\ell ^{\infty}(G)$-valued absolutely
summable sequences on $G$, but we will often interpret it  as the projective tensor
product 
\[\ell ^1(G,\ell ^{\infty}(G)) =  \ell ^1(G) \, \projtensor \,  \ell ^{\infty}(G).\]
Thus for an element $f\in \ell ^1(G,\ell ^{\infty}(G))$ we denote its value in
$\ell ^{\infty}(G)$ by $f(x)$, $ x\in G$, and we  write $f(x)(z)$ or $f(x,z)$ for 
the value of this $\ell ^{\infty}$-function at $z\in G$.

The twisted convolution of $h,f \in \mathcal{L}$  is defined  by
\[
h \conv f (x) =  \sum_{y \in G} T_y h(xy) f(y^{-1}), \;\mbox{ for } x\in G \, ,
\]
and the  involution of $h\in \mathcal{L}$  by
\[
h^{\ast} (x) = \overline{T_{x^{-1}} h(x^{-1})}, \;\mbox{ for } x\in
G\, .
\]
An element $f\in \mathcal{L} $ may be  represented uniquely as
\[f = \sum_{z\in G}m_z\delta_z,\]
where $m_z=f(z)\in \ell ^{\infty}(G)$. By using $m_z$ as the $z$-th
side-diagonal of a matrix,  we define a map 
\begin{equation}\label{eq:R}
R:\ell ^1(G,\ell ^{\infty}(G),T) \to B(\ell ^2(G))
\end{equation}
by
\begin{equation}
  \label{eq:cc6}
Rf = \sum _{z\in G} D^{m_z}_z.
\end{equation}
\begin{prop}\label{p:1} The map
$R:\ell ^1(G,\ell ^{\infty}(G),T) \to CD(G)$ is an iso\-metric $\ast$-isomorphism.
\end{prop}
\begin{proof}
Let  $f=\sum _{z\in G}m_z\delta_z$ and $h=\sum _{z\in G} n_z\delta_z
\in \cL$. By \eqref{eq:c1}  
we have
\begin{eqnarray*}
\|Rf\|_{CD} &=& \norm[1]{\sum _{z\in G} D^{m_z}_z} \\
 &=& \sum _{z\in G} \sup_{\{x,y : xy^{-1}=z\}} |D^{m_z}_z(x,y)| \\
&=& \sum _{z\in G} \sup_y \abs{m_z(y)} \;=\; \norm[\ell ^1(G,\ell ^{\infty}(G))]{f}.
\end{eqnarray*}
Thus $R$ is an isometry. 

The  twisted convolution of $f$ and $h$ may be computed as follows: 
\begin{eqnarray}
(h\conv f )(x,z) &=& \big[ \sum _{y\in G} T_y h(xy) f(y^{-1})\big](z)\notag \\
&=&\sum _{y\in G} h(xy,y^{-1}z)f(y^{-1},z) \notag\\
&=&\sum _{y\in G} n_{xy}(y^{-1}z) m_{y^{-1}}(z) \label{eq:cc7}\\
&=&\sum_v l_v(z)\delta_v(x) \notag
\end{eqnarray}
where 
$$
l_v= \sum _{y\in G} T_y n_{vy}\, m_{y^{-1}} \in \ell ^\infty (G) \, .
$$
By comparison, the composition of the corresponding   operators $Rf$
and $Rh$ (matrix  multiplication) yields that 
\begin{eqnarray*}
Rf \circ Rh &=& \sum_r D^{n_r}_r \circ \sum_w D^{m_w}_w \\
&=& \sum_{r,w} D^{(T_{w^{-1}}n_r)\,m_w}_{rw}
= \sum_v D^{l'_v}_v
\end{eqnarray*}
where 
\[
l'_v =\sum_{\{r,w:\, rw=v\}} T_{w^{-1}}n_r\,m_w = \sum _{y\in G}
T_{y}n_{vy}\,m_{y^{-1}}=l_v.
\]
Thus  $Rf \circ Rh = R(f \conv h )$ and $R$ is an algebra homomorphism.
\par
The involution of  $f$ as above is given by 
\begin{eqnarray*}
f^{\ast}(x,z)&=& \overline{T_{x^{-1}}m_{x^{-1}}(z)}\\
&=& \overline{m_{x^{-1}}(xz)}\\
&=&\sum_v l_v(z)\delta_v(x),
\end{eqnarray*}
where $l_v(z)=\overline{T_{v^{-1}}m_{v^{-1}}}(z)$.\\
By comparison, the adjoint of  a single side-diagonal operator is 
\begin{eqnarray*}
(D^{m_v}_v)^{\ast}(x,y)&=& \overline{D^{m_v}_v(y,x)}\\
&=& \overline{m_v (x)}\delta_v(yx^{-1})\\
&=&\overline{m_v (x)}\delta_{v^{-1}}(xy^{-1})\\
&=&\overline{m_v (v^{-1}y)}\delta_{v^{-1}}(xy^{-1})\\
&=& D^{T_v \overline{m_v}}_{v^{-1}}(x,y)\, .
\end{eqnarray*}
These equalities imply that 
\begin{eqnarray*}
(\sum_v D^{m_v}_v)^{\ast}&=& \sum_v D^{T_v \overline{m_v}}_{v^{-1}}\\
&=&\sum_v D^{T_{v^{-1}} \overline{m_{v^{-1}}}}_{v}\;=\;R(f^{\ast})\, ,
\end{eqnarray*}
and so $R$ preserves the involution. Finally, from the definition of $\|A\|_1$
and the equalities~\eqref{eq:c1} and \eqref{eq:cc6} one sees that $R$ is onto. 
\end{proof}
\section{Symmetry of the Twisted $L^1$-Algebra}\label{sec:3}
Recall that a Banach algebra $A$ with isometric involution is called 
symmetric if the spectrum of  every positive element is   contained in the
non-negative reals, i.~e.~$\spr{a^{\ast}a}\subset [0,\infty)$ for
all $ a\in A$. For  various abstract characterizations of symmetry
see~\cite[Section~41]{BD73} or \cite{LP79}. 

Furthermore, a locally compact group $G$ is called symmetric,  if its
convolution algebra $L^1(G)$ is symmetric. Various classes of groups are known to be
symmetric: (a) locally compact  Abelian  groups, (b) compact groups,
(c) finite extensions of discrete
nilpotent groups, (d)  compactly generated groups of polynomial
growth, (e) compact extensions of locally compact nilpotent groups,
and others. See~\cite{Lep80}. For  the groups of the 
classes (a) --- (c)  Leptin and Poguntke~\cite{LP79} have  shown that
they   satisfy an even  stronger property, namely that  of rigid
symmetry. This means that 
for every  $C^{\ast}$-algebra $C$ the projective tensor product
$L^1(G)\projtensor C$ is symmetric. Later Poguntke~\cite{pog92} showed
that all nilpotent 
locally compact groups are rigidly symmetric.

Our goal is to show that 
the twisted $L^1$-algebra $\mathcal{L}=\ell ^1(G,\ell ^{\infty}(G),T)$
of  a rigidly symmetric discrete group $G$  is symmetric  and hence
that the  algebra of convolution-dominated operators $CD(G)$ is 
also symmetric.
\par
To this end we define a map 
\begin{equation}\label{eq:Q}
Q:\ell ^1(G,\ell ^{\infty}(G),T) \to \ell ^1(G) \projtensor B(\ell ^2(G))
\end{equation}
by 
\begin{equation}
  \label{eq:c23}
  f=\sum_v \delta_v \otimes m_v \mapsto \sum_v \delta_v \otimes
  D^{m_v}_v\, .
\end{equation}

\begin{prop}\label{p:2}
The  map $Q$ is an isometric $\ast$-iso\-mor\-phism of 
$\ell ^1(G,\ell ^{\infty}(G),T)$ onto a closed subalgebra of $\ell ^1(G) \projtensor
B(\ell ^2(G))$.
\end{prop}
\begin{proof}
The proof rests on the isometrical identification $\ell ^1(G,E)=\ell
^1(G)\, \projtensor \,  E$, which holds for any Banach space $E$~\cite[Ch.~VIII.1.]{DU77}. 
It follows that
for $f=\sum_v \delta_v\otimes m_v \in \mathcal{L}$
\begin{eqnarray*}
\norm[1]{f} &=& \sum_v \norm[\infty]{m_v}\; = \; 
\sum_v\norm[B(\ell ^2(G))]{D^{m_v}_v}\\
&=& \norm[\ell ^1(G)\projtensor B(\ell ^2(G))] {\sum_v \delta_v\otimes D^{m_v}_v}.
\end{eqnarray*}
Thus $Q$ is an isometry. 
Let $h=\sum_v \delta_v\otimes n_v$,  then  by \eqref{eq:cc7}
\[
h\conv f = \sum_v \delta_v\otimes l_v,
\]
where
$l_v= \sum _{y\in G} (T_y n_{vy} )m_{y^{-1}}$. Hence
\begin{eqnarray*}
Q(h\conv f) &=& \sum_v \delta_v\otimes D^{l_v}_v\\
&=&\sum_v \delta_v\otimes\sum_{\{z,w:zw=v\} }D^{n_z}_z D^{m_w}_w\\
&=& \sum_{z,w }\delta_z \delta_w \otimes D^{n_z}_z D^{m_w}_w\\
&=&(\sum _{z\in G} \delta_z \otimes D^{n_z}_z )(\sum_w\delta_w \otimes D^{m_w}_w)%
\;=\; Q(h)  Q(f)\, .
\end{eqnarray*}
Similarly one computes that $Q$ intertwines the involutions.
In fact
\begin{eqnarray*}
  Q(f)^{\ast}&=&\sum_v \delta_v^{\ast}\otimes (D^{m_v}_v)^{\ast}\\
&=&\sum_v \delta_{v^{-1}}\otimes D^{T_v \overline{m_v}}_{v^{-1}}\\
&=&\sum_{v^{-1}} \delta_{v}\otimes D^{T_{v^{-1}}
  \overline{m_{v^{-1}}}}_{v}  = Q(f^*) \, .
\end{eqnarray*}
Thus $Q$ is a $\ast $-homomorphism. Since $Q$ is an  isometry, the
image of $Q$ is a closed subalgebra of $\ell ^1 
\, \projtensor \, \cB (\ell ^2)$. 
\end{proof}
Since symmetry is inherited by   closed subalgebras,  we   obtain the
following consequence. 
\begin{cor}\label{cor:1}
Let $G$ be a discrete rigidly symmetric group. Then 
 $\ell ^1(G,\ell ^{\infty}(G),T)$ and $CD(G)$ are symmetric Banach $\ast$-algebras.
\end{cor}
\section{Inverse Closedness}

Given two Banach algebras $\cA \subseteq \cB $ with common identity,
$\cA $ is inverse-closed in $\cB $, if 
$$a\in \cA  \, \text{ and } a^{-1} \in \cB \,\, \Rightarrow \,\, a^{-1} \in
\cA \, .
$$
This notion occurs under many names: one also says that $\cA $ is a
spectral subalgebra or a local subalgebra  of $\cB $.  The pair $(\cA ,
\cB )$ is called a Wiener pair by Naimark~\cite{naimark72}.   An
important property of an inverse-closed subalgebra $\cA $ is that it
possesses the  same holomorphic functional calculus   as $\cB $. 

Inverse-closedness is usually proved by means of Hulanicki's
Lemma~\cite{hul72}.  Let $r_{\cA }(a) $ denote the spectral radius of
$a$ in the algebra $\cA $. If $r_{\cA }(a) = r_{\cB }(a) $ for all $a=
a^* \in \cA $, then we have equality of the spectra $\mathrm{sp}_{\cA
} (a) =  \mathrm{sp}_{\cB} (a) $ for all $a\in \cA $. Consequently, if
$\cB $ is symmetric, then $\cA $ is also symmetric. 
For this version of Hulanicki's lemma,  see~\cite[Lemma 3.1 and
6.1]{fglm} and \cite[Lemma~5.1]{grocomp} for an elementary proof.

Our goal is to show that the algebra of convolution-dominated matrices
$CD(G)$ 
is inverse-closed in $\mathcal{B}(\ell ^2(G))$.  For this purpose   we
consider two natural  unitary representations of the twisted 
$L^1$-algebra $\cL $.
\par
The first representation  is  the so-called  $D$-regular representation of $\cL
$. Recall that  $D : m \mapsto D^m$ is a faithful representation of
the $C^{\ast}$-algebra $\ell ^{\infty}$  by multiplication operators in $B(\ell ^2(G))$.
Then as in Leptin~\cite[\S 3]{lep68} the  $D$-regular representation
 $\lambda^D$ of $\mathcal{L}=\ell ^1(G,\ell ^{\infty},T)$ 
on the Hilbert space $\ell ^2(G,\ell ^2(G))$
is defined  by
\[
\lambda^D(f)\xi (x) = \sum _{y\in G} D^{T_y f(xy)} \xi(y^{-1}),\;\xi\in \ell ^2(G,\ell ^2(G)),\;f\in\mathcal{L} .
\]
One easily verifies that this defines indeed  a $\ast$-representation of
$\mathcal{L}$.
\par
The second representation is the mapping
$R:\mathcal{L}\to CD(G)\subset B(\ell ^2(G))$ introduced
in~\eqref{eq:cc6}. By Proposition~\ref{p:1},  $R$  is also  
 a $\ast$-representation of $\mathcal{L}$ on $\ell ^2(G)$. We call
 this representation the 
 canonical representation of $\mathcal{L}$.  

\begin{prop}\label{prop:regular}
The $D$-regular representation $\lambda ^D$ of $\mathcal{L}$ is a multiple of the
canonical representation $R$. Hence $\|R(f)\| = \|\lambda ^D(f)\|$ for
all $f\in \mathcal{L}$.
\end{prop}
\begin{proof}
We identify $\ell ^2(G,\ell ^2(G))$ with $\ell ^2(G\times G)$. Let
$R^\omega $ be the extension of $R$ from $\ell ^2(G)$ to $\ell
^2(G\times G)$  by letting the operators 
$R(f)=\sum _{y\in G} \lambda(y)\circ D^{f(y)},\; f\in \mathcal{L}$, 
act in the first coordinate only, i.e., for $\xi \in \ell ^2(G\times
G)$ 
\begin{equation}
  R^{\omega}(f) \xi (x,z) =\sum _{y\in G} f(y)(y^{-1}x)\xi(y^{-1}x,z).
\end{equation}
Next we  define
a  candidate for an intertwining operator between the  $D$-regular
representation and the  $\card{G}$-multiple $R^\omega $ of the canonical representation
by
\[
S\xi(x,z)=\xi(xz,z),\;\mbox{ where }\xi\in \ell ^2(G\times G).
\]
Then on the one hand we have 
\[
S[R^{\omega}(f) \xi](x,z) = \sum _{y\in G} f(y)(y^{-1}xz)\xi(y^{-1}xz,z).
\]
 On the other hand
\begin{eqnarray*}
\lambda^D(f) (S\xi) (x,z) &=& \sum _{y\in G} (T_yf(xy))(z) (S\xi)(y^{-1},z)\\
&=& \sum _{y\in G} (T_{x^{-1}y}f(y))(z) (S\xi)(y^{-1}x,z)\\
&=& \sum _{y\in G} f(y)(y^{-1}xz) (S\xi)(y^{-1}x,z)\\
&=& \sum _{y\in G} f(y)(y^{-1}xz) \xi(y^{-1}xz,z).
\end{eqnarray*}
Consequently, 
\begin{equation}
  \label{eq:cc9}
\lambda ^D(f) (S\xi ) = SR^\omega (f)\xi   
\end{equation}
 for all $f\in
\cL $ and $\xi \in \ell ^2(G\times G)$. Since $S $ is unitary on $\ell
^2(G\times G)$, $\lambda ^D$ and $R^\omega $ are equivalent.
\end{proof}

To deal with inverse-closedness, we need to compare several norms on
$\mathcal{L}$ and $CD(G)$. Let $\norm[\ast]{\;.\;}$ be  the largest
$C^{\ast}$  norm on $\mathcal{L}$. By a theorem of Ptak~\cite{ptak70}
a Banach $*$-algebra $\cA $  is symmetric,
\fif\ the largest $C^*$-seminorm $\| \cdot \|_*$ on $\cA $ satisfies
$\|a^*a\|_* = r_{\cA } (a^*a)$ for all $a\in \cA $. See also~\cite[\S
41 Corollary 8]{BD73}.

As a first consequence of Proposition~\ref{prop:regular} we identify
the largest $C^*$-norm on $CD(G)$. 

\begin{cor}\label{cor:2}
Let $G$ be an amenable discrete group, then the largest  $C^{\ast}$ norm on
$\mathcal{L}$ equals the operator norm on $CD(G)$.
\end{cor}
\begin{proof}
Since $G$ is amenable, it follows from \cite[Satz 6]{lep68} of Leptin
that for the representation $D$ of $\ell ^{\infty}(G)$ the $D$-regular
representation $\lambda^D$ defines the largest  $C^{\ast}$ norm on $\mathcal{L}$.
Therefore  we obtain 
\[
\norm[\ast]{f} = \norm{\lambda^D(f)} = \norm[B(\ell ^2(G))]{R(f)}
\qquad  \text{for every }  f\in \mathcal{L}\, ,
\]
where the last equality follows from Proposition~\ref{prop:regular}.
\end{proof}

\begin{prop} \label{prop:normid}
Let $G$ be a discrete, amenable, and rigidly symmetric group.
Then 
\begin{equation}
  \label{eq:cc10}
  \sprad[\mathcal{L}]{f^{\ast}f} = \sprad[CD(G)]{R(f)^{\ast}R(f)} =
\norm[B(L^2(G))]{R(f)}^2 \qquad \text{ for all } f\in \mathcal{L} .
\end{equation}
\end{prop}
\begin{proof}
Since $\mathcal{L}$ and $CD(G)$ are symmetric by
Corollary~\ref{cor:1}, Ptaks theorem~\cite{ptak70} implies  that
$\norm[\ast]{f}^2=\sprad[\mathcal{L}]{f^{\ast}f}=\sprad[CD(G)]{R(f)^{\ast}R(f)}$.
Since  Corollary~\ref{cor:2} says that $\|f^*f\|_* =
 \|R(f)^*R(f)\|_{\cB(\ell ^2)}$, we obtain the identity~\eqref{eq:cc10}. 
\end{proof}

\begin{thm}\label{thm:invclosed}
Let $G$ be a discrete, amenable, and rigidly symmetric group. If $f\in
\mathcal{L}$
is such that $R(f)\in CD(G)$ has an inverse in $B(\ell ^2(G))$ then $f^{-1}$
exists in $\mathcal{L}$ and $R(f^{-1})=R(f)^{-1}$ is in $CD(G)$.
\end{thm}
\begin{proof}
If $f\in\mathcal{L}$ is hermitian, i.e. $f=f^{\ast}$, then by
Proposition~\ref{prop:normid} 
\[
\sprad[\mathcal{L}]{f}^2= \sprad[\mathcal{L}]{f^{\ast}f} =\norm[B(L^2(G))]{R(f)}^2 .\]
\cite[Lemma 6.1 and 3.1]{fglm} imply
that
 \[ \spr[\mathcal{L}]{f} = \spr[B(\ell ^2(G))]{R(f)},\quad 
 \forall f\in\mathcal{L}.\]
Thus the invertibility of $R(f)$ in $\mathcal{B}(\ell ^2(G))$ implies
the invertibility of $f$ in $\mathcal{L}$. 
\end{proof}

By writing  Theorem~\ref{thm:invclosed}  explicitly as a statement
about the off-diagonal decay of an invertible matrix, we recover
Theorem~\ref{main} of the introduction.

\begin{cor}
  \label{main1}
Let $G$ be a discrete, amenable, and rigidly symmetric group (for
instance,  a finitely generated group of polynomial growth).   If a
matrix $A$ indexed by $G$ satisfies the off-diagonal 
decay condition 
$|A(x,y)| \leq a(xy^{-1}) $ for some $a\in \ell ^1(G)$ and $A$ is
invertible on $\ell ^2(G)$, then there exists  $b\in \ell ^1(G)$ such
that $|A^{-1}(x,y)| \leq b(xy^{-1})$. 
\end{cor}

A slight variation yields the following result of which previous
versions have been quite useful in time-frequency
analysis~\cite{fg97jfa}.

\begin{cor}\label{main2}
  Assume that $A \in CD(G)$ and that $A=A^*$. Then the following are
  equivalent: 

(i)  $A$ is invertible on $\ell ^2(G)$. 

(ii)  $A$ is invertible on $\ell ^p(G)$ for \emph{all} $p, 1\leq p \leq
\infty $. 

(iii)   $A$ is invertible on $\ell ^p(G)$ for \emph{some} $p, 1\leq p \leq
\infty $.   
\end{cor}

\begin{proof}
(i) $\, \Rightarrow \, $ (ii)  Recall that every  matrix $A\in CD(G)$ is
bounded on all $\ell ^p (G)$, 
  $1\leq p \leq \infty $ by \eqref{eq:cc28}. 
Thus if $A\in CD(G)$ is
  invertible on $\ell ^2(G)$, then by Theorem~\ref{main} $A^{-1} \in
  CD(G)$ and thus $A^{-1}$ is invertible on $\ell ^p(G)$ for arbitrary
  $p, 1\leq p\leq \infty $. The implication (ii) $\, \Rightarrow \, $
  (iii) is obvious. 

(iii) $\, \Rightarrow \, $ (i) Assume that $A$ is invertible on some $\ell ^p(G)$. Then the adjoint
operator $A^*=A$ is invertible on the dual space $\ell ^{p'}(G)$,
where $p'= p/(p-1)$ is the conjugate index. By interpolation we obtain
that $A$ is invertible  on the interpolation space $\ell ^2(G)$.
\end{proof}

\noindent \textsl{REMARK:}
The hypotheses on the group $G$ are almost sharp. To see this, let
$\lambda (f) $ denote the convolution operator $c \mapsto \lambda (f) c =
f \ast c$ acting on $\ell ^p (G)$, and let $\spr[\ell ^p]{f}$ the
spectrum of $\lambda (f) $ as an operator acting on $\ell ^p(G)$. Then
$\spr[\ell ^p]{f}= \spr[\ell ^2]{f}$ for all $p\in [1,\infty ]$, if
and only if the group $G$ is amenable and
symmetric~\cite{barnes90,GL04}. Thus amenability and symmetric are
necessary in Theorem~\ref{thm:invclosed}. We do not know whether we
can replace the rigid symmetry by symmetry, because it is an open
problem whether every symmetric group is rigidly
symmetric~\cite{pog92}. 

We emphasize once more that all discrete finitely generated groups of
polynomial growth satisfy the hypotheses of amenability and rigid
symmetry. These groups are    finite extensions of some  discrete
nilpotent  group by Gromov's result~\cite{grom81}, and thus they are
rigidly symmetric by ~\cite{LP79} and~\cite{pog92}.

\section{Symmetry of weighted algebras}

In this section we extend the results about the symmetry of
convolution-domi\-nated operators to the weighted case.

A function $\w :G\to [1,\infty)$ is called a weight on $G$, if it  fulfills
\begin{eqnarray*}
\w(xy) &\leq& \w(x)\w(y),\quad\forall x,y\in G\\
\w(x^{-1})&=&\w(x),\quad\forall x\in G\\
\w(e)&=&1 \, .
\end{eqnarray*}
Given a weight $\omega $ we let
$\lw{G,}{\w}$ be the weighted $\ell ^1$-algebra on $G$.
Using weights, one can  model stronger decay conditions
on convolution-dominated operators as follows.
\begin{defi} An operator $A$ on $\ell ^2(G)$ is called
  \w-convolution-dominated, $A\in CD(G,\omega )$ in short,  if there
  exists an  
  $a\in \lw{G,}{\w}$ such that 
\[
\abs{A(x,y)} \leq a(xy^{-1}), \quad \forall x,y\in G.
\]
We define its norm as
\[
\norm[\w]{A} := \inf \{ \norm[\lw{G,}{\w}]{a} : a\in \lw{G,}{\w},\; \abs{A(x,y)} \leq
a(xy^{-1})\;\forall x,y\in G\}.
\]
\end{defi}
As in  the unweighted case, we may write the norm as 
\[
\|A\|_\omega = \sum_{z\in G} \quad \sup_{\{x,y:\,xy^{-1}=z\}}
\abs{A(x,y)} \quad \w(z)\; < \; \infty\, . 
\]
\par
Thus an operator $A$ is in $CD(G,\omega )$, if it is dominated by a
convolution operator in $\ell ^1(G,\omega)$ in the sense that 
$|Ac(x)| \leq ( a \ast |c|)(x)$ for some $a\in \ell ^1(G,\omega )$. 
Since  $\lw{G,}{\w}$ is a  convolution algebra,   
the space of
\w-convolution-dominated operators $CD(G,\omega )$ is a Banach
$\ast$-algebra with respect to  composition of operators 
and the usual involution of operators in $B(\ell ^2(G))$. Furthermore,
$CD(G,\omega ) \subseteq CD(G) \subseteq \mathcal{B}(\ell ^2(G))$. 
\par
For the study of $CD(G,\omega )$, we consider the weighted, twisted $L^1$-algebra 
$\mathcal{L}_{\w}=\lw{G,\w,}{\ell ^{\infty}(G),T}$, which is defined
as a subalgebra of  $\mathcal{L}$
endowed  with the norm
$$
\|f\|_{\mathcal{L}_{\w}} = \sum _{x\in G} \|f(x) \|_\infty \, \w (x)
\, .
$$


Since $\cL _{\w} $ is a subalgebra of $\cL $,  all algebraical
relations are preserved and the results of Sections~2 and~3 carry over 
to $\cL _{\w}$ after a slight modification of the  norm computations.
\begin{prop} \label{pr:1}
 Let  $R_{\w}$ and $Q_{\w}$ denote  the restrictions of the maps
$R$ and  $Q$ defined in (\ref{eq:R}) and (\ref{eq:Q}) from
$\mathcal{L}$  to
$\mathcal{L}_\omega  = \lw{G,\w,}{\ell ^{\infty}(G),T}$.
Then
\[
R_{\w}:\ell ^1(G,\w,\ell ^{\infty}(G),T) \to CD(G,\w)
\]
 is an iso\-metric  $\ast$-isomorphism and
\[
Q_{\w}: \lw{G,\w,}{\ell ^{\infty}(G),T} \to
\lw{G,}{\w}\projtensor B(\ell ^2(G))
\]
 is an isometric $\ast$-isomorphism
onto a closed $*$-subalgebra.
\end{prop}
\par
We are interested in  the symmetry of the weighted
$\ell ^1$-algebra.  This forces us to impose some conditions of subexponential
growth on the weight.
\begin{defi} \label{grrs}
(a) A weight $\omega$ is said to satisfy the \emph{\grs-condition}
(Gelfand-Raikov-Shilov condition)  if
\[
\lim_{n\to\infty} \w(x^n)^{1/n} =1 \qquad \text{ for all } x\in G.
\]
(b) A weight \w\ is said to  satisfy  the 
\emph{\ugrs-condition} (the uniform \grs-condition), if for some
generating subset $U$ of $G$ containing the identity element
\[
\lim _{n\to \infty } \sup _{y\in U^n} \w(y) ^{1/n} = \lim _{n\to \infty } \sup _{x_1,
    \dots , x_n \in U} \w(x_1 x_2 \dots x_n ) ^{1/n} = 1 \, .
\]
\end{defi}

The \grs-condition is a necessary condition for the spectral identity
$\sprad[\ell ^1    ]{f} = \sprad[\ell ^1_\omega   ]{f}$ in weighted
group algebras, and hence for the symmetry of $\ell ^1(G,\omega )$~\cite{fgl}. If
$G$ is a   compactly generated   locally compact 
group of polynomial growth, then the \grs-condition is also sufficient
for  the symmetry of $\ell ^1(G,\omega )$. In this case, the 
\ugrs-condition with  a relatively compact  set $U$  is also   equivalent to
the \grs-condition by the 
results in ~\cite{fgl}. 
However,  if $G$ is not compactly generated, 
the \ugrs-condition may be  a stronger assumption  on the weight.

 We
emphasize that in Definition~\ref{grrs}, $U$ need not be
finite. As a example consider the group $\bZ ^2$ and the weight
$\omega (k_1,k_2) = (1+|k_1|)^s, k_1, k_2 \in \bZ, s>0$. This weight
satisfies the \grs-condition and the \ugrs-condition with the
generating set $\{-1,0,1\} \times \bZ $. 

\begin{thm}\label{th:2}
Let  $G$ be  a rigidly symmetric, amenable,  discrete group. If  the
weight \w\ satisfies the \ugrs-condition and the condition
\begin{equation}
  \label{eq:cc24}
  \sup _{x\in U^n \setminus U^{n-1}} w(x) \leq C   \inf _{x\in U^n
    \setminus U^{n-1}} w(x) \, ,
\end{equation} 
then $\lw{G,}{\w}\projtensor B(\ell ^2(G))$ is inverse-closed in
$\ell ^1(G) \projtensor \cB  (\ell ^2(G))$ and hence symmetric.
\end{thm}
\begin{proof}
By the assumption on $G$ we know that the algebra
$\mathcal{B}=\lw{G}{}\, \projtensor \, B(\ell ^2(G))$ is symmetric.
Since $\mathcal{A}=\lw{G,}{\w}\, \projtensor \,  B(\ell ^2(G))$ is a
subalgebra of $\mathcal{B}$,
by  \cite[Lemmas~3.1 and 6.1]{fglm}, we need only show the
equality of the both 
spectral radii on the latter algebra.

Since for $f\in \mathcal{A}$
\[
\norm[\mathcal{B}]{f}=\sum_{x\in G} \norm[B(\ell ^2(G))]{f(x)}%
\leq\sum_{x\in G} \norm[B(\ell ^2(G))]{f(x)}\w(x)=%
\norm[\mathcal{A}]{f},
\]
the spectral radius formula implies  that 
\[
\sprad[\mathcal{B}]{f}\leq \sprad[\mathcal{A}]{f} \qquad \text{ for
  all }  f\in \mathcal{A}.
\]
Thus  it suffices to show the converse inequality.
To this end we define a weight $v$ on $\ZZ$ by
\[
v (n)  =  \sup _{y\in U^{|n|}} \w (y),
\]
where $U$ is a generating set, containing the identity element, such
that \linebreak[4]
$\lim _{n\to \infty } \sup _{y\in U^n} \w(y) ^{1/n} =1$. 
By induction one finds an estimation for the norm of the 
n-th convolution power $f^{(n)}$ of $f\in \mathcal{A}$:
\begin{equation}
  \label{eq:2}
  \norm[\mathcal{A}]{f^{(n)} } \leq \sum_G \dots \sum_G \norm{f(x_1)}\, 
\norm{f(x_2)}\,
  \dots \norm{f(x_n)}\, \w(x_1 \dots x_n) \,.  
\end{equation}
Since $G= \bigcup _{n=1}^\infty \big(U^n \setminus U^{n-1}\big)$ as a
 disjoint union (where $U^0 = \emptyset$), we may split each sum
accordingly. This yields  
\begin{eqnarray*}
\lefteqn{\norm[\mathcal{A}]{f^{(n)}}\;\leq\;}\\ 
&\leq& \sum _{k_1, k_2 , \dots , k_n =1}^\infty \sum_{U^{k_1}\setminus
U^{k_1-1}} \dots \sum_{U^{k_n}\setminus U^{k_n-1}}  
\norm{f(x_1)}\dots \norm{f(x_n)}\, \w(x_1 \dots x_n)\, .
\end{eqnarray*}
If $x_j \in U^{k_j}\setminus   U^{k_j-1}$, then $x_1 \dots x_n \in
  U^{k_1 + \dots + k_n }$ 
  and so the weight is majorized by 
\[
 \w(x_1 \dots x_n) \leq \sup _{y\in   U^{k_1 + \dots + k_n }} \w(y) =
 v (k_1 + \dots + k_n) \, .
\]
Set $b_k := \sum_{U^k \setminus U^{k-1}} \norm{f(x)} $ and 
$b=(b_k)_{k\in \NN}$. Then
clearly we have $\norm[\mathcal{B}]{f} = \norm[\ell ^1]{b}$ and
condition~\eqref{eq:cc24}
implies that $C^{-1} \|b\|_{\ell ^1_v} \leq \|f\|_{\mathcal{L}_\omega
} \leq \|b\|_{\ell ^1_v} $. For the convolution powers of $f$ we
obtain that 
\[
\norm[\mathcal{A}]{f^{(n)}}\leq \sum _{k_1, k_2 , \dots , k_n =1}^\infty
b_{k_1} b_{k_2} \dots b_{k_n} v (k_1 + k_2 + \dots k_n ) =
\norm[{\lw{\ZZ,}{v}}]{b^{(n)} } <\infty \, .
\]
By its  definition  the weight $v$ on $\ZZ$ satisfies the
\grs-condition, and $\lw{\ZZ,}{v}$ is symmetric
by~\cite[Lemma~3.2]{fgl}. Hence
\begin{eqnarray*}
  \sprad[\mathcal{A}]{f} &=& 
\lim _{n\to \infty } \norm[\mathcal{A}]{f^{(n)}} ^{1/n} \leq  
\lim _{n\to \infty }\norm[\lw{\ZZ,}{v}]{b^{(n)} }^{1/n}\\
&=&
\sprad[\lw{\ZZ,}{v}]{b}
\;=\;  \sprad[\lw{\ZZ}{}]{b}\;=\;\norm[\ell ^1]{b}\\
& =&  \norm[\mathcal{B}]{f} \, .
\end{eqnarray*}
So for all $k\in \NN $ we have
\[
\sprad[\mathcal{A}]{f}= \sprad[\mathcal{A}]{ f^{(k)}}^{1/k} \leq  
\norm[\mathcal{B}]{f^{(k)}}^{1/k} \, ,
\]
and by letting $k\to \infty $ we obtain the required inequality
$\sprad[\mathcal{A}]{f} \leq \sprad[\mathcal{B}]{f}$.
\end{proof}

Combining  Proposition~\ref{pr:1} and  Theorem~\ref{th:2},  we obtain
the symmetry of the weight\-ed convolution-dominated operator algebras
$CD(G,\omega )$.  
\begin{cor}
Under the same assumptions on $G$ and $\w$ as in Theorem~\ref{th:2}, the algebra $CD(G,\w)$ is symmetric.
\end{cor}
Moreover,  the  Theorem~\ref{th:2} combined with Theorem~\ref{thm:invclosed} shows  that for $f\in
\mathcal{L}_{\w}$:
\begin{eqnarray*}
\sprad[\mathcal{L}_{\w}]{f}&\;=\;& \sprad[\mathcal{A}]{Q_{\w}(f)} \; =
\; \sprad[\mathcal{B}]{Q(f)}\\
&=&
\sprad[\mathcal{L}]{f}\;=\;
\sprad[B(\ell ^2(G))]{R(f)}= \sprad[B(\ell ^2(G))]{R_{\w}(f)}.
\end{eqnarray*}
Using  Hulanicki's Lemma in the form of \cite[Lemma 6.1 and 3.1]{fglm} we conclude as in the proof of 
Theorem~\ref{thm:invclosed} that $CD(G,\w)$ is inverse-closed in
$B(\ell ^2(G))$. 
\begin{cor}\label{cor:gw}
Impose the same assumptions on $G$ and $\w$ as in Theorem~\ref{th:2}. 

If $f\in
\mathcal{L}_{\w}$
is such that $R_{\w}(f)\in CD(G,\w)$ has an inverse in $B(\ell ^2(G))$ 
then $f^{-1}$
exists in $\mathcal{L}_{\w}$, and $R_{\w}(f^{-1})=R_{\w}(f)^{-1}$ 
is in $CD(G,\w)$.
\end{cor}

For a single matrix Corollary~\ref{cor:gw}  can be recast  once again
as a statement about the preservation of the off-diagonal decay by the
inverse. 

\begin{cor}
Impose the  same assumptions on $G$ and $\omega $ as in
Theorem~\ref{th:2}. 
If a matrix $A$ on $G$ satisfies the off-diagonal decay condition
$|A(x,y)| \leq a(xy^{-1}), \forall x,y\in G,$ for some $a\in \ell ^1(G,\omega)$ and $A$
is invertible on $\ell ^2(G)$, then there exists some $b\in \ell
^1(G,\omega)$, such that $|A^{-1}(x,y)| \leq b(xy^{-1}), \forall x,y
\in G$. 
\end{cor}

\noindent \textsl{REMARK:} The   proof of Thm.~\ref{th:2} is similar to
the one of \cite[Thm.~3.3]{fgl}. However, the proof given there works only under an
additional assumption on the weight, such as \eqref{eq:cc24}, the result remains
correct as a consequence of the main result in \cite{fgl}. 

\bibliographystyle{abbrv}

\end{document}